\documentclass[12pt]{amsart}
\usepackage{amsfonts}
\usepackage{amssymb,amsmath,amscd}
\usepackage{epsfig}
\usepackage{graphicx}

\newtheorem{thm}{\sc Theorem}[section]
\newtheorem{prop}[thm]{\sc Proposition}
\newtheorem{lem}[thm]{\sc Lemma}
\newtheorem{cor}[thm]{\sc Corollary}
\theoremstyle{definition}

\theoremstyle{definition}
\newtheorem{de}[thm]{\sc Definition}
\theoremstyle{definition}
\newtheorem{rem}[thm]{\sc Remark}
\theoremstyle{definition}

\numberwithin{equation}{section}

\DeclareMathOperator{\R}{\mathbb R}

\begin{document}
\title[The Fundamental Theorems]{The Fundamental Theorems for curves and
surfaces in 3D Heisenberg group}
\author[H.-L.~Chiu, S.-H. Lai]{HUNG-LIN~CHIU and SIN-HUA LAI}
\address{Department of Mathematics, National Central University, Chung Li,
32054, Taiwan, R.O.C.}
\email{hlchiu@math.ncu.edu.tw}
\address{Department of Mathematics, National Central University, Chung Li,
32054, Taiwan, R.O.C.}
\email{972401001@cc.ncu.edu.tw}

\begin{abstract}
We study the local equivalence problems of curves and surfaces in
3-dimensional Heisenberg group via Cartan's method of moving frames and Lie
groups, and find a complete set of invariants for curves and surfaces. For
surfaces, in terms of these invariants and their suitable derivatives, we
also give a Gaussian curvature fromula of the metric induced from the
adapted metric on $H^1$, and hence form a new formula for the Euler number
of a closed surface.
\end{abstract}

\maketitle

%$\thanks{Email: hlchiu@math.ncu.edu.tw}
%\keywords{}\subjclass{Primary 32V05, 32V20; Secondary 53C56.}
\renewcommand{\subjclassname}{\textup{2000} Mathematics Subject
Classification}
%\keywords{Pseudo-hermitian manifold, sub-Laplacian, eigenvalues, CR Paneitz operator.}

%\tableofcontents

\section{Introduction}

\label{sec:1} In $3$-dimensional Euclidean space, it is well known that any
unit-speed curve is completely determined by its curvature and torsion. This
means that given any two function $k(s)$ and $\tau (s)$ with $k(s)>0$
everywhere, then there exists a unit-speed curve whose curvature and torsion
are $k$ and $\tau $, respectively. In addition, such a unit-speed curve is
unique up to a Euclidean rigid motion. This is the fundamental theorem of
curves. On the other hand, the fundamental theorem of surfaces says that,
instead of the scalar-invariants, the first and second fundamental forms are
the complete invariants for surfaces. In this paper we will show that there
are the analogous fundamental theorems of curves and surfaces in $3$%
-dimensional Heisenberg group $H^{1}$.

The Heisenberg group $H^{1}$ is the space $\R^{3}$ associated with the group
multiplication
\begin{equation}
(x_{1},y_{1},z_{1})\circ
(x_{2},y_{2},z_{2})=(x_{1}+x_{2},y_{1}+y_{2},z_{1}+z_{2}+y_{1}x_{2}-x_{1}y_{2}).
\end{equation}%
It is a $3$-dimensional Lie group. The space of all left invariant vector
fields is spanned by the following three vector fields:
\begin{equation}
\mathring{e}_{1}=\frac{\partial }{\partial x}+y\frac{\partial }{\partial z}%
,\  \  \mathring{e}_{2}=\frac{\partial }{\partial y}-x\frac{\partial }{%
\partial z}\  \  \text{and}\  \ T=\frac{\partial }{\partial z}.
\end{equation}%
The standard contact bundle on $H^{1}$ is the subbundle $\xi _{0}$ of the
tangent bundle $TH^{1}$ which is spanned by $e_{1}$ and $e_{2}$. It is also
defined to be the kernel of the contact form
\begin{equation}
\theta _{0}=dz+xdy-ydx.
\end{equation}%
The CR structure on $H^{1}$ is the endomorphism $J_{0}:\xi _{0}\rightarrow
\xi _{0}$ defined by
\begin{equation}
J_{0}(\mathring{e}_{1})=\mathring{e}_{2}\  \  \text{and}\ J_{0}(\mathring{e}%
_{2})=-\mathring{e}_{1}.
\end{equation}%
We sometimes view the Heisenberg group $H^{1}$ as a pseudohermitian manifold
when we consider it together with the standard pseudo-hermitian structure $%
(J_{0},\theta _{0})$. For the details about pseudo-hermitian structure, we
refer the readers to \cite{Le1}, \cite{Le2}, or \cite{We}. Let $PSH(1)$ be
the group of Heisenberg rigid motions, that is, the group of all
pseudo-hermitian transformations on $H^{1}$. Recall that a pseudo-hermitian
transformation on $H^{1}$ is a diffeomorphism on $H^{1}$ which preserves the
standard pseudo-hermitian structure $(J_{0},\theta _{0})$. In Subsection \ref%
{pseutran}, we give an explicit expression for a pseudo-hermitian
transformation.

Let $\gamma:(a,b)\rightarrow H^1$ be a parametrized curve. For each $%
t\in(a,b)$, the velocity $\gamma^{^{\prime }}(t)$ of $\gamma(t)$ has the
natural decompostion
\begin{equation}
\gamma^{^{\prime }}(t)=\gamma^{^{\prime }}_{\xi_{0}}(t)+\gamma^{^{\prime
}}_{T}(t),
\end{equation}
where $\gamma^{^{\prime }}_{\xi_{0}}(t)$ and $\gamma^{^{\prime }}_{T}(t)$
are, respectively, the orthogonal projection of $\gamma^{^{\prime }}(t)$ on $%
\xi_{0}$ along $T$ and the orthogonal projection of $\gamma^{^{\prime }}(t)$
on $T$ along $\xi_{0}$.

\begin{de}
A \textbf{horizontally regular curve} is a parametrized curve $\gamma(t)$
such that $\gamma^{^{\prime }}_{\xi_{0}}(t)\neq 0$ for each $t\in(a,b)$.
\end{de}

Proposition \ref{norpara} shows that a horizontally regular curve can always
be reparametrized by a parameter $s$ such that $|\gamma _{\xi
_{0}}^{^{\prime }}(s)|=1$ for every $s$. We call such a paramter $s$ the
horizontal arc-length, which is unique up to a constant.

For a horizontally regular curve $\gamma(s)$ parametrized by the horizontal
arc-length $s$, we define the $p$-curvature $k(s)$ and $T$-variation $%
\tau(s) $ as
\begin{equation}
\begin{split}
k(s)&=<\frac{dX(s)}{ds},Y(s)> \\
\tau(s)&=<\gamma^{^{\prime }}(s),T>,
\end{split}%
\end{equation}
where $X(s)=\gamma^{^{\prime }}_{\xi_{0}}(s)$ and $Y(s)=J_{0}X(s)$. We have
the following fundamental theorem for curves in $H^1$ which says that
horizontally regular curves are completely prescribed by the $p$-curvature
and $T$-variation as well.

\begin{thm}
\label{main1} Let $\gamma_{1}(s)$ and $\gamma_{2}(s)$ be two horizontally
regular curves parametrized by the horizontal arc-length. Suppose that they
have the same $p$-curvature $k(s)$ and $T$-variation $\tau(s)$. Then there
exists $g\in PSH(1)$ such that
\begin{equation}
\gamma_{2}(s)=g\circ \gamma_{1}(s),\  \text{for all}\ s.
\end{equation}
In addition, given smooth functions $k(s), \tau(s)$, there exists a
horizontally regular curve $\gamma(s)$, parametrized by the horizontal
arc-length, having $k(s)$ and $\tau(s)$ as its $p$-curvature and $T$%
-variation, respectively.
\end{thm}

We say $\gamma (t)$ is a horizontal curve if $\gamma ^{^{\prime }}(t)=\gamma
_{\xi _{0}}^{^{\prime }}(t)$ for each $t\in (a,b)$. By the previous
definition $\gamma (s)$ is horizontal if and only if the $T$-variation $\tau
(s)=0$, we have immediately the following corollary.

\begin{cor}
Let $\gamma_{1}(s)$ and $\gamma_{2}(s)$ be two horizontal unit-speed curves
in $H^1$ with the same $p$-curvature $k(s)$. Then there exists $g\in PSH(1)$
such that
\begin{equation}
\gamma_{2}(s)=g\circ \gamma_{1}(s),\  \text{for all}\ s.
\end{equation}
In addition, given a smooth function $k(s)$, there exists a horizontal
unit-speed curve $\gamma(s)$ having $k(s)$ as its $p$-curvature.
\end{cor}

In Subsection \ref{computofpt}, we compute the explicit formulae for the $p$%
-curvature and $T$-variation and get the following theorem.

\begin{thm}
\label{main2} Let $\gamma(t)=\big(x(t),y(t),z(t)\big)\in H^1$ be a
horizontally regular curve, not necessarily parametrized by horizontal
arc-length. Then the $p$-curvature $k(t)$ and $T$-variation $\tau(t)$ are
having the forms
\begin{equation}  \label{curformula}
\begin{split}
k(t)&=\frac{x^{^{\prime }}y^{^{\prime \prime }}-x^{^{\prime \prime
}}y^{^{\prime }}}{\big((x^{^{\prime }})^{2}+(y^{^{\prime }})^{2}\big)^{\frac{%
3}{2}}}(t) \\
\tau(t)&=\frac{xy^{^{\prime }}-x^{^{\prime }}y+z^{^{\prime }}}{\big(%
(x^{^{\prime }})^{2}+(y^{^{\prime }})^{2}\big)^{\frac{1}{2}}}(t).
\end{split}%
\end{equation}
\end{thm}

As an application, we proceed to compute the $p$-curvature and $T$-variation
of the geodesics of $H^1$ in Subsection \ref{computofpt} and obtain a
characteristic description of the geodesics in $H^1$.

\begin{thm}
\label{chaofgeo} The geodesics of $H^1$ are just those horizontally regular
curves with vanishing $T$-variation and constant $p$-curvature, that is, $%
\tau=0$ and $k=c$ for some constant $c\in \R$.
\end{thm}

Observing the formula (\ref{curformula}), which says that the $p$-curvature
of $\gamma(t)=\big(x(t),y(t),z(t)\big)$ is just the signed curvature of the
plane curve $\alpha(t)=\pi \circ \gamma(t)=\big(x(t),y(t)\big)$, where $\pi$
is the projection on $xy$ plane along the $z$-axis. On the other hand, it is
well known that the signed curvature completely describes the plane curves,
therefore we have immediately the following corollary:

\begin{cor}
If two horizontally regular curves in $H^{1}$ differ by a Heisenberg rigid
motion then their projections on $xy$-plane differ by a Euclidean rigid motion. In
particular, two horizontal curves in $H^{1}$ differ by a Heisenberg rigid
motion if and only if their projections on $xy$-plane are congruent in the Euclidean plane.
\end{cor}

For a surface $\Sigma \subset H^{1}$ which is embedded in $H^{1}$, we can
also say something about fundamental theorem. First of all, we recall that a
singular point $p\in \Sigma $ is a point such that, at $p$, the tangent
plane $T_{p}\Sigma $ coincides with the contact plane $\xi _{0}(p)$.
Therefore outside the singular set (the non-singular part of $\Sigma $), it
is integrated to be a one-dimensional foliation for the intersection of $%
T\Sigma $ and $\xi _{0}$, which is called the characteristic foliation. Now
we define the normal coordinates.

\begin{de}
Let $F:U\rightarrow H^1$ be a parametrized surface with coordinates $(u,v)$
on $U\subset \R^2$. We say $F$ is normal if

\begin{enumerate}
\item $F(U)$ is a surface without singular points;

\item $F_{u}=\frac{\partial F}{\partial u}$ defines the characteristic
foliation on $F(U)$;

\item $|F_{u}|=1$ for each point $(u,v)\in U$, where the norm is respect to
the levi-metric on $H^1$.
\end{enumerate}

We call $(u,v)$ a normal coordinates.
\end{de}

It is easy to see that every non-singular point $p\in \Sigma$, there exists
a normal coordinates around $p$. For a normal parametrized surface $%
F:U\rightarrow H^1$, let $X=F_{u},\ Y=J_{0}X$ and $T=\frac{\partial}{%
\partial z}$, we define functions $a, b, c, l$ and $m$ on $U$ by
\begin{equation}  \label{coeofform}
\begin{array}{lll}
a=<F_{v},X> & b=<F_{v},Y> & c=<F_{v},T> \\
l=<F_{uu},Y> & m=<F_{uv},Y>. &
\end{array}%
\end{equation}
They satisfy the integrability conditions
\begin{equation}  \label{intcons1}
\begin{split}
a_{u}&=bl,\  \ b_{u}=-al+m,\  \ c_{u}=2b \\
l_{v}&-m_{u}=0.
\end{split}%
\end{equation}

The following theorem says that these functions are complete differential
invariants for the map $F$. We call $a, b$ and $c$ the coefficients of the
first kind of $F$, and $l, m$ the second kind.

\begin{thm}
\label{main4} Let $U\subset \R^{2}$ be a simply connected open set. Suppose
that $a,b,c,l$ and $m$ are functions on $U$ which satisfy the integrability
condition (\ref{intcons1}). Then there exists a normal parametrized surface $%
F:U\rightarrow H^{1}$ having $a,b,c$ and $l,m$ as the coefficients of the
first kind and the second kind, respectively. In addition, if $\widetilde{F}%
:U\rightarrow H^{1}$ is another such a normal parametrized surface, then it
differs from $F$ by a Heisenberg rigid motion, that is, there exists a
motion $g\in PSH(1)$ such that $\widetilde{F}(u,v)=g\circ F(u,v)$ for all $%
(u,v)\in U$.
\end{thm}

Note that, from (\ref{tlofc2}), we see that $l$, up to a sign, is
independent of the choice of the normal coordinates, hence it is a
differential invariant of the surface $F(U)$. Actually $l$ is the $p$-mean
curvature. Therefore $l=0$ means that $F(U)$ is a $p$-minimal surface. Such
a parametrization $F:U\rightarrow H^{1}$ is called a normal $p$-minimal
parametrized surface. From the integrability condition (\ref{intcons1}), we
see that the second kind of coefficient $m$ is entirely determined by the
first kind as
\begin{equation}
m=b_{u}.  \label{secbyfir1}
\end{equation}%
The integrability conditions (\ref{intcons1}) hence become to be
\begin{equation}
a_{u}=0,\  \ b_{uu}=0,\  \ c_{u}=2b,  \label{intcons3}
\end{equation}%
and thus we obtain the following corollary of Theorem \ref{main4}.

\begin{thm}
\label{main5} Let $U\subset \R^{2}$ be a simply connected open set. Suppose
that $a,b$ and $c$ are three functions on $U$ which satisfy the
integrability condition (\ref{intcons3}). Then there exists a normal $p$%
-minimal parametrized surface $F:U\rightarrow H^{1}$ having $a,b$ and $c$ as
the first kind of coefficients of $F$, and the second kind of coefficient is
determined by $b$ as (\ref{secbyfir1}). In addition, if $\widetilde{F}%
:U\rightarrow H^{1}$ is another such a normal $p$-minimal parametrized
surface, then it differs from $F$ by a Heisenberg rigid motion, that is,
there exists a motion $g\in PSH(1)$ such that $\widetilde{F}(u,v)=g\circ
F(u,v)$ for all $(u,v)\in U$.
\end{thm}

Besides the $p$-mean curvature $l$, in Section \ref{invpasur}, we also show
that both $\alpha =\frac{b}{c}$, up to a sign (which is called the $p$%
-variation), and the adapted metric $g_{\theta _{0}}$ restricted to the
surface are also invariants of the surface $F(U)$. Actually $\alpha $ is the
function such that the vector field $\alpha e_{2}+T$ is tangent to the
surface, where $e_{2}=J_{0}e_{1}$ and $e_{1}$ is a unit vector field tangnet
to the characteristic foliation. Let $e_{\Sigma }$ be another unit vector
field tangent to the surface which is defined by
\begin{equation*}
e_{\Sigma }=\frac{\alpha e_{2}+T}{\sqrt{1+\alpha ^{2}}}.
\end{equation*}%
We have that these three invariants satisfy the integrability condition:
\begin{equation}
\begin{split}
(1+\alpha ^{2})^{\frac{3}{2}}(e_{\Sigma }l)& =(1+\alpha
^{2})(e_{1}e_{1}\alpha )-\alpha (e_{1}\alpha )^{2}+4\alpha (1+\alpha
^{2})(e_{1}\alpha ) \\
& +\alpha (1+\alpha ^{2})^{2}K+\alpha l(1+\alpha ^{2})^{\frac{1}{2}%
}(e_{\Sigma }\alpha )+\alpha (1+\alpha ^{2})l^{2},
\end{split}
\label{Intconsur}
\end{equation}%
where $K$ is the Gaussian curvature with respect to $g_{\theta
_{0}}|_{\Sigma }$.

The following theorem says that the Riemannian metric induced from the
adapted metric together with the $p$-mean curvature $l$ and $p$-variation $%
\alpha$ is a complete system of invariants for a surface without singular
point.

\begin{thm}[The fundamental theorem for surfaces in $H^{1}$]
\label{main8} Let $(\Sigma ,g)$ be a Riemannian $2$-manifold with Guassian
curvature $K$, and let $\alpha ^{^{\prime }},l^{^{\prime }}$ be two
real-valued functions on $\Sigma $. Assume that $K$, together with $\alpha
^{^{\prime }}$ and $l^{^{\prime }}$, satisfy the integrability condition (%
\ref{Intconsur}), with $\alpha ,l$ replaced by $\alpha ^{^{\prime
}},l^{^{\prime }}$, respectively. Then for every point $x\in
\Sigma $ there exists an open neighborhood $U$ containing $x$, and an
embedding $f:U\rightarrow H^{1}$ such that $g=f^{\ast }(g_{\theta
_{0}}),\alpha ^{^{\prime }}=f^{\ast }\alpha $ and $l^{^{\prime }}=f^{\ast }l$%
, where $\alpha ,l$ are the induced $p$-variation and $p$-mean curvature on $%
f(U)$. Moreover, $f$ is unique up to a Heisenberg rigid motion.
\end{thm}

In the proof of Theorem \ref{main8}, we also get

\begin{thm}
\label{main6} Let $\Sigma \subset H^1$ be an oriented surface.Then the
Gaussian curvature $K$ of the restricted metric $g_{\theta_{0}}|_{\Sigma}$
can be expressed by means of $l,\alpha$ and the derivatives of $\alpha$.
\begin{equation}  \label{Gaussfor}
K=\frac{(e_{1}\alpha)^{2}+2(1+\alpha^{2})(e_{1}\alpha)+4\alpha^{2}(1+%
\alpha^{2})-l(e_{\Sigma}\alpha)(1+\alpha^{2})^{\frac{1}{2}}} {%
(1+\alpha^{2})^{2}}.
\end{equation}
\end{thm}

By the Gauss-Bonnet formula, we immediately have the following corollary.

\begin{thm}
\label{main7} Let $\Sigma \subset H^1$ be a closed, oriented surface. Then
we have
\begin{equation}
\begin{split}
2\pi \chi(\Sigma)&=\int_{\Sigma}\frac{(e_{1}\alpha)^{2}+2(1+%
\alpha^{2})(e_{1}\alpha)+4\alpha^{2}(1+\alpha^{2})-l(e_{\Sigma}\alpha)(1+%
\alpha^{2})^{\frac{1}{2}}} {(1+\alpha^{2})^{2}} d\sigma \\
&=\int_{\Sigma}\frac{(e_{1}\alpha)^{2}+2(1+\alpha^{2})(e_{1}\alpha)+4%
\alpha^{2}(1+\alpha^{2})-l(e_{\Sigma}\alpha)(1+\alpha^{2})^{\frac{1}{2}}} {%
(1+\alpha^{2})^{\frac{3}{2}}} \omega^{1}\wedge \theta_{0},
\end{split}%
\end{equation}
where $d\sigma $ is the area form with respect to the induced metric from
the adapted metric $g_{\theta_{0}}$, and $\chi(\Sigma)$ is the Euler number
of $\Sigma$.
\end{thm}

Substituting the Gaussian curvature formula (\ref{Gaussfor}) into (\ref%
{Intconsur}), we see that the integrability condition (\ref{Intconsur}) is
equivalent to the Gaussian equation (\ref{Gaussfor}) together with the following
Codazzi-Like equation:
\begin{equation}
e_{\Sigma}l=\frac{e_{1}e_{1}\alpha+6\alpha(e_{1}\alpha)+4\alpha^{3}+\alpha
l^2}{(1+\alpha^2)^{\frac{1}{2}}}.
\end{equation}

\begin{rem}
There is also an integrability condition for a surface expressed as a graph of a function $u$, 
which is called a Codazzi-Like equation and shown up in \cite{CHMY}.
\end{rem}

We now give a brief outline of this paper. In section $2$, we state the two
propositions about uniqueness and existence of mappings of a smooth manifold
into a Lie group G which underlie the theory. In section $3$, we obtain the
representation of $\ PSH(1)$ which is the group of pseudohermitian
transformations on $H^{1}$. Also we discuss how the matrix Lie group $PSH(1)$
interpret as the set of "frames" on the homogeneous space $H^{1}=PSH(1)/SO(2)$.
Then from the (left-invariant) Maurer-Cartan form, we immediately get the
moving frame formula. In section $4$, we compute the Darboux derivative of a
lift of a horizontally regular curve in $H^{1}$ and then to get the
fundamental theorem for curves in $H^{1}$. Moreover, we compute the $p$%
-curvature and the $T$-variation of a horizontally regular curve and
geodesics in\ $H^{1}$. In section $5$, we compute the Darboux derivative of
the lift of a normal parametrized surface. Then we get complete differential
invariants for a normal parametrized surface. In section $6$, let $\Sigma $
be an oriented surface and $f:\Sigma \rightarrow H^{1}$ be an embedding. We
compute the Darboux derivative of the lifting of $f$ to get the fundamental
theorem for surfaces in $H^{1}$. In this section, we also compute the Gaussian formula 
(\ref{Gaussfor}) and the integrability condition (\ref{Intconsur}). Finally, in section $7$, we
give another proof for Theorem \ref{main1}.

\textbf{Acknowledgment.} The first author's research was supported in part
by NCTS and in part by NSC 100-2628-M-008-001-MY4. He would like to thank
Prof. Jih-Hsin Cheng for his teaching and talking on this
topic, Prof. Paul Yang for his encouragement and advising in the research
for the last few years. The second author would like to express her thanks
to Prof. Shu-Cheng Chang for his teaching, constant encouragement and
supports.

\section{Calculus on Lie group}

Let $M$ be a connected smooth manifold, and let $G\subset GL(n,R)$ be a
matrix Lie group with Lie algebra $\mathfrak{g}$ and the (left-invariant)
Maurer-Cartan form $\omega $. In this section, we shall give, without
proofs, two simple and essential local results concerning smooth maps from a
manifold $M$ into a Lie group $G$. These two results play a fundamental role
in whole of the paper. For the details, we refer the readers to \cite{G},%
\cite{IL},\cite{S} and \cite{CC}. The first of these is

\begin{thm}
\label{ft1} Given two maps $f, \widetilde{f}:M\rightarrow G$, then $%
\widetilde{f}^{*}\omega=f^{*}\omega$ if and only if $\widetilde{f}=g\cdot f$
for some $g\in G$.
\end{thm}

The Lie algebra one-form $f^{*}\omega$ is usually called \textbf{the Darboux
derivative} of the map $f:M\rightarrow G$. The second one is a well-known
existence theorem:

\begin{thm}
\label{ft2} Suppose that $\phi$ is a $\mathfrak{g}$-valued one form on a
simply connected manifold $M$. Then there exists a map $f:M\rightarrow G$
with $f^{*}\omega=\phi$ if and only if $d\phi=-\phi \wedge \phi$.\newline
Moreover, the resulting map $f$ is unique up to a group action.
\end{thm}

The proof of Theorem \ref{ft2} is strongly dependent on the Frobenius
theorem.

\section{The group of pseudohermitian transformations on $H^1$}

\subsection{The pseudohermitian transformations on $H^1$}

\label{pseutran} A pseudohermitian transformation on $H^{1}$ is a
diffeomorphism $\Phi $ on $H^{1}$ which preserves both the CR structure $%
J_{0}$ and the contact form $\theta _{0}$, that is, it satisfies
\begin{equation}
\Phi _{\ast }J_{0}=J_{0}\Phi _{\ast }\  \  \text{on}\  \xi _{0}\  \  \text{and}\
\  \Phi ^{\ast }\theta _{0}=\theta _{0}.
\end{equation}%
Let $L_{p}$ be the left translation by $p$ on the Heisenberg group $H^{1}$.
It is easy to see that $L_{p}$ is a pseudohermitian transformation. We give
another pseudohermitian transformation $\Phi _{R}:H^{1}\rightarrow H^{1}$
which is defined by
\begin{equation}
\Phi _{R}\left(
\begin{array}{c}
x \\
y \\
z%
\end{array}%
\right) \longrightarrow \left(
\begin{array}{cc}
R & 0 \\
0 & 1%
\end{array}%
\right) \left(
\begin{array}{c}
x \\
y \\
z%
\end{array}%
\right) ,
\end{equation}%
where $R\in SO(2)$ is a $2\times 2$ orthogonal matrix.

Let $PSH(1)$ be the group of pseudohermitian transformations on $H^{1}$. The
following theorem specifies that the group $PSH(1)$ consists exactly of all
the transformations of the forms $\Phi _{p,R}\doteq L_{p}\circ \Phi _{R}$,
that is, a transformation $\Phi _{R}$ followed by a left translation $L_{p}$%
. We have
\begin{equation}
\Phi _{p,R}\left(
\begin{array}{c}
x \\
y \\
z%
\end{array}%
\right) =\left(
\begin{array}{c}
ax+by+p_{1} \\
cx+dy+p_{2} \\
(ap_{2}-cp_{1})x+(bp_{2}-dp_{1})y+z+p_{3}%
\end{array}%
\right) ,
\end{equation}%
where $p=(p_{1},p_{2},p_{3})^{t}\in H^{1}$ and $R=\left(
\begin{array}{cc}
a & b \\
c & d%
\end{array}%
\right) \in SO(2)$.

\begin{thm}
\label{psgrp} Let $\phi:H^{1}\rightarrow H^{1}$ be a pseudohermitian
transformation. Then $\Phi=L_{p}\circ \Phi_{R}$ for some $R\in SO(2)$ and $%
p\in H^{1}$.
\end{thm}

\begin{proof}
Let $\Phi:H^{1}\rightarrow H^{1}$ be a pseudohermitian transformation such
that $\Phi(0)=p$. Then the composition $L_{p^{-1}}\circ \Phi$ is a
transformation fixing the origion. Therefore, we reduce the proof of Theorem %
\ref{psgrp} to prove that any pseudohermitian transformation $\Phi$ with $%
\Phi(0)=0$ has the form $\Phi=\Phi_{R}$ for some $R\in SO(2)$. This is
equivalent to prove the following Lemma:

\begin{lem}
\label{balemma} Let $\Phi $ be a pseudohermitian transformation on $H^{1}$
such that $\Phi (0)=0$. Then, for any $p\in H^{1}$, the matrix
representation of $\Phi _{\ast }(p)$ with respect to $(\frac{\partial }{%
\partial x},\frac{\partial }{\partial y},\frac{\partial }{\partial z})$ is
\begin{equation}
\Phi _{\ast }(p)=%
\begin{pmatrix}
\cos \alpha _{0} & -\sin \alpha _{0} & 0 \\
\sin \alpha _{0} & \cos \alpha _{0} & 0 \\
0 & 0 & 1%
\end{pmatrix}%
_{\left( \frac{\partial }{\partial x},\frac{\partial }{\partial y},\frac{%
\partial }{\partial z}\right) },  \label{matrep}
\end{equation}%
for some real constant $\alpha _{0}$ which is independent of $p$. That is $%
\Phi _{\ast }$ is a constant matrix.
\end{lem}

Now we prove Lemma \ref{balemma}. First we compute the matrix representation
of $\Phi _{\ast }(p)$ with respect to $(\mathring{e}_{1},\mathring{e}_{2},T=%
\frac{\partial }{\partial z})$. Since, for $i=1,2,$
\begin{equation*}
\theta _{0}\left( \Phi _{\ast }\mathring{e}_{i}\right) =\left( \Phi ^{\ast
}\theta _{0}\right) \left( \mathring{e}_{i}\right) =\theta _{0}\left(
\mathring{e}_{i}\right) =0,
\end{equation*}%
we see that $\xi _{0}$ is invariant under $\Phi _{\ast }$. Furthermore, let $%
h$ be the Levi metric on $\xi _{0}$ defined by $h(X,Y)=d\theta
_{0}(X,J_{0}Y) $. We have
\begin{equation*}
\begin{split}
\Phi ^{\ast }h(X,Y)& =h(\Phi _{\ast }X,\Phi _{\ast }Y)=d\theta _{0}(\Phi
_{\ast }X,J_{0}\Phi _{\ast }Y) \\
& =d\theta _{0}(\Phi _{\ast }X,\Phi _{\ast }J_{0}Y)=\Phi ^{\ast }(d\theta
_{0})(X,J_{0}Y)=d(\Phi ^{\ast }\theta _{0})(X,J_{0}Y) \\
& =d\theta _{0}(X,J_{0}Y)=h(X,Y).
\end{split}%
\end{equation*}%
That is $h\left( \Phi _{\ast }X,\Phi _{\ast }Y\right) =h\left( X,Y\right) $
for every $X,Y\in \xi _{0}=\ker \theta _{0}$. Thus $\Phi _{\ast }$ is
orthogonal on $\xi _{0}$. On the other hand,
\begin{equation*}
\theta _{0}(\Phi _{\ast }T)=\theta _{0}\left( \Phi _{\ast }\frac{\partial }{%
\partial z}\right) =\left( \Phi ^{\ast }\theta _{0}\right) \left( \frac{%
\partial }{\partial z}\right) =\theta _{0}\left( \frac{\partial }{\partial z}%
\right) =1,
\end{equation*}%
and, for all $X\in \xi _{0}$,
\begin{equation*}
\begin{split}
d\theta _{0}(X,\Phi _{\ast }T)& =d\theta _{0}(\Phi _{\ast }\Phi _{\ast
}^{-1}X,\Phi _{\ast }T)=(\Phi ^{\ast }d\theta _{0})(\Phi _{\ast }^{-1}X,T) \\
& =(d\Phi ^{\ast }\theta _{0})(\Phi _{\ast }^{-1}X,T)=d\theta _{0}(\Phi
_{\ast }^{-1}X,T)=0.
\end{split}%
\end{equation*}%
By the uniqueness of the characteristic vector field, we have $\Phi _{\ast
}T=T$. From the above argument, we conclude that
\begin{equation*}
\Phi _{\ast }(p)=%
\begin{pmatrix}
\cos \alpha \left( p\right) & -\sin \alpha \left( p\right) & 0 \\
\sin \alpha \left( p\right) & \cos \alpha \left( p\right) & 0 \\
0 & 0 & 1%
\end{pmatrix}%
_{\left( \mathring{e}_{1},\mathring{e}_{2},\frac{\partial }{\partial z}%
\right) },
\end{equation*}%
for some real valued function $\alpha $ on $H^{1}$.

Next, let $\Phi =(\Phi ^{1},\Phi ^{2},\Phi ^{3})$, we would like to change
the matrix representation of $\Phi _{\ast }(p)$ from $\left( \mathring{e}%
_{1},\mathring{e}_{2},\frac{\partial }{\partial z}\right) $ to $\left( \frac{%
\partial }{\partial x},\frac{\partial }{\partial y},\frac{\partial }{%
\partial z}\right) $. Let $p=\left( p_{1},p_{2},p_{3}\right) $, $\mathring{e}%
_{1}\left( p\right) =\frac{\partial }{\partial x}+p_{2}\frac{\partial }{%
\partial z}$ and $\mathring{e}_{2}\left( p\right) =\frac{\partial }{\partial
y}-p_{1}\frac{\partial }{\partial z}$, then

\begin{equation*}
\begin{split}
\Phi _{\ast }(p)\left( \frac{\partial }{\partial x}\right) & =\Phi _{\ast
}(p)\left[ \mathring{e}_{1}\left( p\right) -p_{2}\frac{\partial }{\partial z}%
\right] =\Phi _{\ast }(p)\left[ \mathring{e}_{1}\left( p\right) \right]
-p_{2}\frac{\partial }{\partial z} \\
& =\cos \alpha \left( p\right) \mathring{e}_{1}\left[ \Phi (p)\right] +\sin
\alpha \left( p\right) \mathring{e}_{2}\left[ \Phi (p)\right] -p_{2}\frac{%
\partial }{\partial z} \\
& =\cos \alpha \left( p\right) \frac{\partial }{\partial x}+\sin \alpha
\left( p\right) \frac{\partial }{\partial y} \\
& \  \  \ +\left[ \cos \alpha \left( p\right) \Phi ^{2}\left( p\right) -\sin
\alpha \left( p\right) \Phi ^{1}\left( p\right) -p_{2}\right] \frac{\partial
}{\partial z},
\end{split}%
\end{equation*}%
and
\begin{equation*}
\begin{split}
\Phi _{\ast }(p)\left( \frac{\partial }{\partial y}\right) & =\Phi _{\ast
}(p)\left[ \mathring{e}_{2}\left( p\right) +p_{1}\frac{\partial }{\partial z}%
\right] =\Phi _{\ast }(p)\left[ \mathring{e}_{2}\left( p\right) \right]
+p_{1}\frac{\partial }{\partial z} \\
& =-\sin \alpha \left( p\right) \mathring{e}_{1}\left[ \Phi (p)\right] +\cos
\alpha \left( p\right) \mathring{e}_{2}\left[ \Phi (p)\right] +p_{1}\frac{%
\partial }{\partial z} \\
& =-\sin \alpha \left( p\right) \frac{\partial }{\partial x}+\cos \alpha
\left( p\right) \frac{\partial }{\partial y} \\
& \  \  \ +\left[ -\sin \alpha \left( p\right) \Phi ^{2}\left( p\right) -\cos
\alpha \left( p\right) \Phi ^{1}\left( p\right) +p_{1}\right] \frac{\partial
}{\partial z}.
\end{split}%
\end{equation*}%
Thus,
\begin{equation}
\Phi _{\ast }(p)=%
\begin{pmatrix}
\cos \alpha \left( p\right) & -\sin \alpha \left( p\right) & 0 \\
\sin \alpha \left( p\right) & \cos \alpha \left( p\right) & 0 \\
\Phi _{x}^{3}(p) & \Phi _{y}^{3}(p) & 1%
\end{pmatrix}%
_{\left( \frac{\partial }{\partial x},\frac{\partial }{\partial y},\frac{%
\partial }{\partial z}\right) ,}  \label{31}
\end{equation}%
where
\begin{equation}
\begin{split}
\Phi _{x}^{3}(p)& =\cos \alpha \left( p\right) \Phi ^{2}\left( p\right)
-\sin \alpha \left( p\right) \Phi ^{1}\left( p\right) -p_{2}, \\
\Phi _{y}^{3}(p)& =-\sin \alpha \left( p\right) \Phi ^{2}\left( p\right)
-\cos \alpha \left( p\right) \Phi ^{1}\left( p\right) +p_{1}.
\end{split}
\label{32}
\end{equation}%
Observing first that, from (\ref{31}), $\Phi _{z}^{1}=\Phi _{z}^{2}=0$, so
both $\Phi ^{1}$ and $\Phi ^{2}$ are function depending only on $x$ and $y$,
hence so is $\alpha $. Secondly, since $\Phi _{xy}^{1}=\Phi _{yx}^{1}$ and $%
\Phi _{xy}^{2}=\Phi _{yx}^{2}$, we have, from (\ref{31}),
\begin{equation*}
\begin{pmatrix}
\cos \alpha & -\sin \alpha \\
\sin \alpha & \cos \alpha%
\end{pmatrix}%
\begin{pmatrix}
\alpha _{x} \\
\alpha _{y}%
\end{pmatrix}%
=%
\begin{pmatrix}
0 \\
0%
\end{pmatrix}%
,
\end{equation*}%
which implies that $\alpha _{x}=\alpha _{y}=0$. Thus $\alpha $ is a constant
on $H^{1}$, say $\alpha =\alpha _{0}$. From (\ref{31}) again and note that $%
\Phi (0)=0$, we have that
\begin{equation*}
\begin{split}
\Phi ^{1}& =x\cos {\alpha _{0}}-y\sin {\alpha _{0}} \\
\Phi ^{2}& =x\sin {\alpha _{0}}+y\cos {\alpha _{0}},
\end{split}%
\end{equation*}%
which implies that $\Phi _{x}^{3}=\Phi _{y}^{3}=0$. Thus
\begin{equation*}
\Phi _{\ast }(p)=%
\begin{pmatrix}
\cos \alpha _{0} & -\sin \alpha _{0} & 0 \\
\sin \alpha _{0} & \cos \alpha _{0} & 0 \\
0 & 0 & 1%
\end{pmatrix}%
_{\left( \frac{\partial }{\partial x},\frac{\partial }{\partial y},\frac{%
\partial }{\partial z}\right) }.
\end{equation*}%
This completes the proof.
\end{proof}

\subsection{Representation of $PSH(1)$}

We can represent $\Phi_{p,R}$ and points of $H^{1}$, respectively, as
\begin{equation}
\Phi_{p,R}\leftrightarrow M=\left(%
\begin{array}{cccc}
1 & 0 & 0 & 0 \\
p_{1} & a & b & 0 \\
p_{2} & c & d & 0 \\
p_{3} & ap_{2}-cp_{1} & bp_{2}-dp_{1} & 1%
\end{array}%
\right),
\end{equation}
and
\begin{equation}
\left(%
\begin{array}{c}
x \\
y \\
z%
\end{array}%
\right)\leftrightarrow X=\left(%
\begin{array}{c}
1 \\
x \\
y \\
z%
\end{array}%
\right).
\end{equation}
Then
\begin{equation}
MX=\left(%
\begin{array}{c}
1 \\
\Phi_{p,R}\left(%
\begin{array}{c}
x \\
y \\
z%
\end{array}%
\right)%
\end{array}%
\right).
\end{equation}
That is, $PSH(1)$ may be represented as a matrix group by writing
\begin{equation}
PSH(1)=\left \{ M\in GL(4,R)\  \Big{|}\  \ M=\left(%
\begin{array}{cccc}
1 & 0 & 0 & 0 \\
p_{1} & a & b & 0 \\
p_{2} & c & d & 0 \\
p_{3} & ap_{2}-cp_{1} & bp_{2}-dp_{1} & 1%
\end{array}%
\right) \right \}.
\end{equation}

Let $psh(1)$ be the Lie algebra of $PSH(1)$. Then it is easy to see that the
element of $psh(1)$ is look as
\begin{equation}
\left(%
\begin{array}{cccc}
0 & 0 & 0 & 0 \\
x_{1} & 0 & -x_{1}{}^{2} & 0 \\
x_{2} & x_{1}{}^{2} & 0 & 0 \\
x_{3} & x_{2} & -x_{1} & 0%
\end{array}%
\right).
\end{equation}
Therefore the Maurer-Cartan form of $PSH(1)$ is look like
\begin{equation}
\omega =\left(%
\begin{array}{cccc}
0 & 0 & 0 & 0 \\
\omega^{1} & 0 & -\omega_{1}{}^{2} & 0 \\
\omega^{2} & \omega_{1}{}^{2} & 0 & 0 \\
\omega^{3} & \omega^{2} & -\omega^{1} & 0%
\end{array}%
\right),
\end{equation}
here $\omega_{1}^{2}$ and $\omega ^{j}, j=1,2,3$ are 1-forms on $PSH(1)$.

\subsection{The oriented frames on $H^1$}

An oriented frame on $H^{1}$ is a frame of the form
\begin{equation}
(p;X,Y,T),
\end{equation}%
where $p\in H^{1},\ Y=J_{0}X$ and $X\in \xi _{0}(p)$ are unit vectors with
respect to the standard levi metric on $H^{1}$. We can also identify $PSH(1)$
with the space of all oriented frames on $H^{1}$ as following:
\begin{equation}
M=\left(
\begin{array}{cccc}
1 & 0 & 0 & 0 \\
p_{1} & a & b & 0 \\
p_{2} & c & d & 0 \\
p_{3} & ap_{2}-cp_{1} & bp_{2}-dp_{1} & 1%
\end{array}%
\right) \leftrightarrow (p;X,Y,T),
\end{equation}%
where
\begin{equation}
\begin{split}
X& =a\frac{\partial }{\partial x}+c\frac{\partial }{\partial y}%
+(ap_{2}-cp_{1})\frac{\partial }{\partial t} \\
Y& =b\frac{\partial }{\partial x}+d\frac{\partial }{\partial y}%
+(bp_{2}-dp_{1})\frac{\partial }{\partial t} \\
p& =(p_{1},p_{2},p_{3})^{t}.
\end{split}%
\end{equation}%
Actually, we have that $X=a\mathring{e}_{1}(p)+c\mathring{e}_{2}(p)$ and $Y=b%
\mathring{e}_{1}(p)+d\mathring{e}_{2}(p)$, hence $M$ is the unique $4\times
4 $ matrix such that
\begin{equation}
(p;X,Y,T)=(0;\mathring{e}_{1},\mathring{e}_{2},T)M.
\end{equation}

\subsection{Moving frame formula}

Since $PSH(1)$ is a matrix Lie group, the Maurer-Cartan form is to be $%
\omega =M^{-1}dM$ or $dM=M\omega $. Thus we immediately get that
\begin{equation}
(dp;dX,dY,dT)=(p;X,Y,T)\left(
\begin{array}{cccc}
0 & 0 & 0 & 0 \\
\omega ^{1} & 0 & -\omega _{1}{}^{2} & 0 \\
\omega ^{2} & \omega _{1}{}^{2} & 0 & 0 \\
\omega ^{3} & \omega ^{2} & -\omega ^{1} & 0%
\end{array}%
\right) ,
\end{equation}%
that is, we have the following moving frame formula:
\begin{equation}
\begin{split}
dp& =X\omega ^{1}+Y\omega ^{2}+T\omega ^{3} \\
dX& =Y\omega _{1}{}^{2}+T\omega ^{2} \\
dY& =-X\omega _{1}{}^{2}-T\omega ^{1} \\
dT& =0.
\end{split}
\label{movingframe}
\end{equation}

\section{Differential invariants of horizontally regular curves in $H^1$}

\begin{prop}
\label{norpara} We can reparametrize a horizontally regular curve $\gamma(t)$
by a horizontal arc-length $s$
\end{prop}

\begin{proof}
Define $s(t)=\int_{0}^{t}|\gamma _{\xi _{0}}^{^{\prime }}(u)|du$. Then any
horizontal arc-length differs $s$ by a constant. By the fundamental theorem
of calculus, we have $\frac{ds}{dt}=|\gamma _{\xi _{0}}^{^{\prime }}(t)|$.
So
\begin{equation}
\frac{d\gamma }{ds}=\frac{d\gamma }{dt}\frac{dt}{ds}=\frac{\gamma ^{^{\prime
}}(t)}{|\gamma _{\xi _{0}}^{^{\prime }}(t)|},
\end{equation}%
hence $\gamma _{\xi _{0}}^{^{\prime }}(s)=\frac{\gamma _{\xi _{0}}^{^{\prime
}}(t)}{|\gamma _{\xi _{0}}^{^{\prime }}(t)|}$, that is $|\gamma _{\xi
_{0}}^{^{\prime }}(s)|=1$.
\end{proof}

\begin{de}
A lift of a mapping $f:M\rightarrow G/H$ is defined to be a map $%
F:M\rightarrow G$ such that the following diagram commutes:

\includegraphics[scale=0.5]{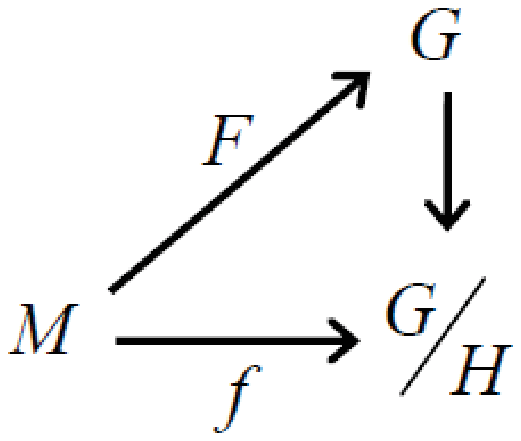}%\\~\\

where $G$ is a Lie group, $H$ is a closed Lie subgroup and $G/H$ is a
homogeneous space. Given a lift $F$ of $f$, any other lift $\tilde{F}%
:M\rightarrow G$ must be of the form
\begin{equation*}
\begin{array}{ccc}
\tilde{F}(x) & = & F(x)g(x)%
\end{array}%
\end{equation*}%
for some map $g:M\rightarrow H$.
\end{de}

\subsection{The Proof of Theorem \protect \ref{main1}}

Let $\gamma (s)$ be a horizontally regular curve with horizontal arc-length
as parameter. For each point of the curve uniquely determines an oriented
frame of $H^{1}$ of the form
\begin{equation}
(\gamma (s);X(s),Y(s),T),
\end{equation}%
where $X(s)=\gamma _{\xi _{0}}^{^{\prime }}(s)$ and $Y(s)=J_{0}X(s)$. Define
$\widetilde{\gamma }(s)$ by
\begin{equation}
\widetilde{\gamma }(s)=(\gamma (s);X(s),Y(s),T).
\end{equation}%
Then $\widetilde{\gamma }(s)$ is a lift of $\gamma (s)$ to $PSH(1)$, which
is uniquely determined by $\gamma (s)$. Let $\omega $ be the Maurer-Cartan
form of $PSH(1)$. We would like to compute the Darboux derivative $%
\widetilde{\gamma }^{\ast }\omega $ of the curve $\widetilde{\gamma }(s)$:

First note that all pull back one-forms by $\widetilde{\gamma }$ are
multiples of $ds$. By (\ref{movingframe}), we have that
\begin{equation}
\begin{split}
d\widetilde{\gamma }(s)& =\widetilde{\gamma }^{\ast }dp \\
& =X(s)\widetilde{\gamma }^{\ast }\omega ^{1}+Y(s)\widetilde{\gamma }^{\ast
}\omega ^{2}+T\widetilde{\gamma }^{\ast }\omega ^{3}.
\end{split}
\label{diffcurve}
\end{equation}%
On the other hand,
\begin{equation}
\begin{split}
d\widetilde{\gamma }(s)& =\gamma _{\xi _{0}}^{^{\prime }}(s)ds+\gamma
_{T}^{^{\prime }}(s)ds \\
& =X(s)ds+\gamma _{T}^{^{\prime }}(s)ds.
\end{split}
\label{diffcurve1}
\end{equation}%
Comparing (\ref{diffcurve}) and (\ref{diffcurve1}), we get
\begin{equation}
\begin{split}
\widetilde{\gamma }^{\ast }\omega ^{1}& =ds,\  \  \widetilde{\gamma }^{\ast
}\omega ^{2}=0 \\
\widetilde{\gamma }^{\ast }\omega ^{3}& =<\gamma ^{^{\prime }}(s),T>ds=\tau
(s)ds.
\end{split}
\label{pulbacform}
\end{equation}%
Again from (\ref{movingframe}), we have
\begin{equation}
dX(s)=Y(s)\widetilde{\gamma }^{\ast }\omega _{1}{}^{2}+T\widetilde{\gamma }%
^{\ast }\omega ^{2}=Y(s)\widetilde{\gamma }^{\ast }\omega _{1}{}^{2},
\end{equation}%
hence
\begin{equation}
\widetilde{\gamma }^{\ast }\omega _{1}{}^{2}=<\frac{dX(s)}{ds}%
,Y(s)>ds=k(s)ds.  \label{pulbacform1}
\end{equation}%
Thus we have already obtained the Darboux derivative of $\widetilde{\gamma }$%
:
\begin{equation}
\widetilde{\gamma }^{\ast }\omega =\left(
\begin{array}{cccc}
0 & 0 & 0 & 0 \\
1 & 0 & -k(s) & 0 \\
0 & k(s) & 0 & 0 \\
\tau (s) & 0 & -1 & 0%
\end{array}%
\right) ds.  \label{fodar}
\end{equation}%
Now suppose that $\gamma _{1}$ and $\gamma _{2}$ have the same $p$-curvature
$k(s)$ and $T$-variation $\tau (s)$. Then, from (\ref{fodar}), we get
\begin{equation*}
\widetilde{\gamma }_{1}^{\ast }\omega =\widetilde{\gamma }_{2}^{\ast }\omega
.
\end{equation*}%
Therefore, by Theorem \ref{ft1}, there exists $g\in PSH(1)$ such that $%
\widetilde{\gamma }_{2}(s)=g\circ \widetilde{\gamma }_{1}(s)$, hence $\gamma
_{2}(s)=g\circ \gamma _{1}(s)$, for all $s$. This completes the uniqueness
up to a group action. To finish the proof of Theorem \ref{main1}, we show
the existence. Given two functions $k(s)$ and $\tau (s)$ defined on an open
interval $I$. Define a $psh(1)$-valued one-form $\varphi $ on $I$ by
\begin{equation*}
\varphi =\left(
\begin{array}{cccc}
0 & 0 & 0 & 0 \\
1 & 0 & -k(s) & 0 \\
0 & k(s) & 0 & 0 \\
\tau (s) & 0 & -1 & 0%
\end{array}%
\right) ds.
\end{equation*}%
Then it is easy to show that $d\varphi +\varphi \wedge \varphi =0$. Thus, by
Theorem \ref{ft2}, there exists a curve
\begin{equation*}
\widetilde{\gamma }(s)=(\gamma (s),X(s),Y(s),T)\in PSH(1)
\end{equation*}%
such that $\widetilde{\gamma }^{\ast }\omega =\varphi $. This means, by
moving frame formula (\ref{movingframe}),
\begin{equation}
\begin{split}
d\gamma (s)& =X(s)ds+\tau (s)Tds \\
dX(s)& =k(s)Y(s)ds \\
dY(s)& =-k(s)X(s)ds-Tds,
\end{split}%
\end{equation}%
which implies that
\begin{equation}
\begin{split}
X(s)& =\gamma _{\xi _{0}}^{^{\prime }}(s),\  \text{and} \\
k(s)& =<\frac{dX(s)}{ds},Y(s)> \\
\tau (s)& =<\frac{d\gamma (s)}{ds},T>.
\end{split}%
\end{equation}%
This completes the proof of the existence.

\subsection{The computation of the $p$-curvature and the $T$-variation}

\label{computofpt} In this subsection, we will compute the $p$-curvature and
the $T$-variation of a horizontally regular curve, and thus give the proof
of Theorem \ref{main2}. After this, we also want to compute the $p$%
-curvature and the $T$-variation of the geodesics of $H^1$. Let $%
\gamma(t)=(x(t),y(t),z(t))$ be a horizontally regular curve. The horizontal
arc-length $s$ is defined by
\begin{equation}
s(t)=\int_{0}^{t}|\gamma_{\xi_{0}}^{^{\prime }}(u)|du,
\end{equation}
where $\gamma_{\xi_{0}}^{^{\prime }}(t)$ is the projection of $%
\gamma^{^{\prime }}(t)$ on $\xi_{0}$ along $T$ direction. Now
\begin{equation}
\begin{split}
\gamma^{^{\prime }}(t)&=(x^{^{\prime }}(t),y^{^{\prime }}(t),z^{^{\prime
}}(t))=x^{^{\prime }}(t)\frac{\partial}{\partial x}+y^{^{\prime }}(t)\frac{%
\partial}{\partial y}+z^{^{\prime }}(t)\frac{\partial}{\partial z} \\
&=x^{^{\prime }}(t)e_{1}+y^{^{\prime }}(t)e_{2}+(z^{^{\prime
}}(t)+xy^{^{\prime }}(t)-yx^{^{\prime }}(t))\frac{\partial}{\partial z},
\end{split}%
\end{equation}
which shows that
\begin{equation}  \label{velexp}
\begin{split}
\gamma_{\xi_{0}}^{^{\prime }}(t)&=x^{^{\prime }}(t)e_{1}+y^{^{\prime
}}(t)e_{2}; \\
\gamma_{T}^{^{\prime }}(t)&=(z^{^{\prime }}(t)+xy^{^{\prime
}}(t)-yx^{^{\prime }}(t))T,
\end{split}%
\end{equation}
where note that $\frac{\partial}{\partial z}=T$. Let $\bar{\gamma}(s)$ be
the reparametrization of $\gamma(t)$ by the horizontal arc-length $s$. Then
we have that $\gamma^{^{\prime }}(t)=\bar{\gamma}^{^{\prime }}(s)\frac{ds}{dt%
}$, hence, comparing with (\ref{velexp}),
\begin{equation}  \label{velexp1}
\begin{split}
\bar{\gamma}_{\xi_{0}}^{^{\prime }}(s)&=\frac{dt}{ds}(x^{^{\prime
}}(t)e_{1}+y^{^{\prime }}(t)e_{2}); \\
\bar{\gamma}_{T}^{^{\prime }}(s)&=\frac{dt}{ds}\left((z^{^{\prime
}}(t)+xy^{^{\prime }}(t)-yx^{^{\prime }}(t))T\right).
\end{split}%
\end{equation}
So the $T$-variation is
\begin{equation}  \label{Tvai}
\begin{split}
\tau(s)&=<\bar{\gamma}^{^{\prime }}(s),T>=<\bar{\gamma}_{T}^{^{\prime
}}(s),T> \\
&=\frac{dt}{ds}(z^{^{\prime }}(t)+xy^{^{\prime }}(t)-yx^{^{\prime }}(t)) \\
&=\frac{xy^{^{\prime }}-x^{^{\prime }}y+z^{^{\prime }}}{\big((x^{^{\prime
}})^{2}+(y^{^{\prime }})^{2}\big)^{\frac{1}{2}}}(t).
\end{split}%
\end{equation}
For the $p$-curvature, first note that $X(s)=\frac{dt}{ds}(x^{^{\prime
}}(t)e_{1}+y^{^{\prime }}(t)e_{2})$, hence $Y(s)=J_{0}X(s)=\frac{dt}{ds}%
(x^{^{\prime }}(t)e_{2}-y^{^{\prime }}(t)e_{1})$. We compute
\begin{equation}
\begin{split}
\frac{dX(s)}{ds}&=\frac{d}{ds}\left(\frac{dt}{ds}\left(x^{^{\prime
}}(t),y^{^{\prime }}(t),x^{^{\prime }}y(t)-xy^{^{\prime }}(t)\right)\right)
\\
&=\left(\frac{dt}{ds}\right)^{2}\left(x^{^{\prime \prime }}(t),y^{^{\prime
\prime }}(t),x^{^{\prime \prime }}y(t)-xy^{^{\prime \prime }}(t)\right)+%
\frac{d^{2}t}{ds^{2}}\left(x^{^{\prime }}(t),y^{^{\prime }}(t),x^{^{\prime
}}y(t)-xy^{^{\prime }}(t)\right) \\
&=\left(x^{^{\prime \prime }}(t)\left(\frac{dt}{ds}\right)^{2}+x^{^{\prime
}}(t)\frac{d^{2}t}{ds^{2}}\right)e_{1}+ \left(y^{^{\prime \prime }}(t)\left(%
\frac{dt}{ds}\right)^{2}+y^{^{\prime }}(t)\frac{d^{2}t}{ds^{2}}\right)e_{2},
\end{split}%
\end{equation}
So
\begin{equation}  \label{pcur}
\begin{split}
k(s)&=<\frac{dX(s)}{ds},Y(s)> \\
&=-\left(x^{^{\prime \prime }}(t)\left(\frac{dt}{ds}\right)^{2}+x^{^{\prime
}}(t)\frac{d^{2}t}{ds^{2}}\right)y^{^{\prime }}(t)\frac{dt}{ds}
+\left(y^{^{\prime \prime }}(t)\left(\frac{dt}{ds}\right)^{2}+y^{^{\prime
}}(t)\frac{d^{2}t}{ds^{2}}\right)x^{^{\prime }}(t)\frac{dt}{ds} \\
&=-\left(x^{^{\prime \prime }}(t)y^{^{\prime }}(t)-x^{^{\prime
}}(t)y^{^{\prime \prime }}(t)\right)\left(\frac{dt}{ds}\right)^3 \\
&=\frac{x^{^{\prime }}y^{^{\prime \prime }}-x^{^{\prime \prime }}y^{^{\prime
}}}{\big((x^{^{\prime }})^{2}+(y^{^{\prime }})^{2}\big)^{\frac{3}{2}}}(t).
\end{split}%
\end{equation}
This completes the proof of Theorem \ref{main2}

Now we make use of (\ref{pcur}) and (\ref{Tvai}) to compute the $p$%
-curvature and $T$-variation of the geodesics in $H^1$. Recall that the
Hamiltonian system on $H^{1}$ for the geodesics is
\begin{equation}  \label{HS}
\begin{array}{lll}
\dot{x}^{k}\left( t\right) & = & h^{kj}\left( x\left( t\right) \right) \xi
_{j}\left( t\right) \\
\dot{\xi}_{k}\left( t\right) & = & -\frac{1}{2}\sum \limits_{i,j=1}^{3}\frac{
\partial h^{ij}\left( x\right) }{\partial x^{k}}\xi _{i}\xi _{j},\text{ }
k=1,2,3,%
\end{array}%
\end{equation}
where
\begin{equation*}
h^{ij}\left( x^{1},x^{2},x^{3}\right) =
\begin{pmatrix}
1 & 0 & x^{2} \\
0 & 1 & -x^{1} \\
x^{2} & -x^{1} & \left( x^{1}\right) ^{2}+\left( x^{2}\right) ^{2}%
\end{pmatrix}
.\qquad
\end{equation*}
So the Hamiltonian system (\ref{HS}) can be expressed by
\begin{equation}
\begin{split}
\dot{x}^{1}\left( t\right) & =\xi _{1}+x^{2}\xi _{3} \\
\dot{x}^{2}\left( t\right) & =\xi _{2}-x^{1}\xi _{3} \\
\dot{x}^{3}\left( t\right) & =x^{2}\xi _{1}-x^{1}\xi _{2}+\xi _{3}\left[
\left( x^{1}\right) ^{2}+\left( x^{2}\right) ^{2}\right] \\
\dot{\xi}_{1}\left( t\right) & =\xi _{2}\xi _{3}-x^{1}\xi _{3}^{2} \\
\dot{\xi}_{2}\left( t\right) & =-\xi _{1}\xi _{3}-x^{2}\xi _{3}^{2} \\
\dot{\xi}_{3}\left( t\right) & =0.
\end{split}%
\end{equation}
Since $\dot{\xi}_{3}\left( t\right) =0,$ thus $\xi _{3}\left( t\right)
=c_{3} $ where $c_{3}$ is some constant. In the case $c_{3}=0$, we have that
$x\left( t\right) =\left( c_{1}t+d_{1},c_{2}t+d_{2},\left(
c_{1}d_{2}-c_{2}d_{1}\right) t+d_{3}\right)$, thus $k\left( t\right) =0$ and
$\tau \left( t\right) =0.$ Next, in the case $c_{3}> 0$, we have
\begin{equation}  \label{gode1}
\begin{split}
x\left( t\right) & = \left( x^{1}\left( t\right) ,x^{2}\left( t\right)
,x^{3}\left( t\right) \right),\  \text{where} \\
x^{1}\left( t\right) & = a_{1}\sin \left( 2c_{3}t\right) +a_{2}\cos \left(
2c_{3}t\right) +d_{1} \\
x^{2}\left( t\right) & = -a_{2}\sin \left( 2c_{3}t\right) +a_{1}\cos \left(
2c_{3}t\right) +d_{2} \\
x^{3}\left( t\right) & = \left( a_{2}d_{1}+a_{1}d_{2}\right) \sin \left(
2c_{3}t\right) +\left( a_{2}d_{2}-a_{1}d_{1}\right) \cos \left(
2c_{3}t\right) \\
& +2c_{3}\left( a_{1}^{2}+a_{2}^{2}\right) t+d_{3},
\end{split}%
\end{equation}
hence $k\left( t\right) =-\frac{1}{\left[ \left( a_{1}^{2}+a_{2}^{2}\right) %
\right] ^{\frac{1}{2}}}<0$ and $\tau \left( t\right) =0.$ Finally, in the
case $c_{3}< 0$, we have
\begin{equation}  \label{gode2}
\begin{split}
x\left( t\right) & = \left( x^{1}\left( t\right) ,x^{2}\left( t\right)
,x^{3}\left( t\right) \right),\  \text{where} \\
x^{1}\left( t\right) & = a_{1}\sin \left(-2c_{3}t\right) +a_{2}\cos \left(-
2c_{3}t\right) +d_{1} \\
x^{2}\left( t\right) & = a_{2}\sin \left(-2c_{3}t\right) -a_{1}\cos \left(
-2c_{3}t\right) +d_{2} \\
x^{3}\left( t\right) & = \left( a_{1}d_{1}+a_{2}d_{2}\right) \sin \left(
-2c_{3}t\right) -\left( a_{2}d_{1}-a_{1}d_{2}\right) \cos \left(
-2c_{3}t\right) \\
& +2c_{3}\left( a_{1}^{2}+a_{2}^{2}\right) t+d_{3},
\end{split}%
\end{equation}
hence $k\left( t\right) =\frac{1}{\left[ \left( a_{1}^{2}+a_{2}^{2}\right) %
\right] ^{\frac{1}{2}}}>0$ and $\tau \left( t\right) =0.$

The above computation shows that a horizontal curve is congruent to a godeic
if it has positive constant $p$-curvature. Conversely, it is easy to see
that a symmetry action of a geodesic is still a geodesic. Therefore we
complete the proof of Theorem \ref{chaofgeo}.

\begin{rem}
Actually, the geodesics (\ref{gode1}) for $c_{3}>0$ are the reverse of the
geodesics (\ref{gode2}) for $c_{3}<0$. That is, they run in the reverse
direction of each other.
\end{rem}

\section{Differential invariants of parametrized surfaces in $H^1$}

\label{invpasur}

\subsection{The proof of Theorem \protect \ref{main4}}

First we show the uniqueness. Let $F:U\rightarrow H^{1}$ be a normal
parametrized surface with $a,b,c,l$ and $m$ as the coefficients. That is,
\begin{equation}
\begin{array}{lll}
a=<F_{v},X> & b=<F_{v},Y> & c=<F_{v},T> \\
l=<F_{uu},Y> & m=<F_{uv},Y>. &
\end{array}%
\end{equation}%
Defining the unique lift $\widetilde{F}$ of $F$ to $PSH(1)$ as
\begin{equation}
\widetilde{F}=<F,X,Y,T>,\ X=F_{u},\ JX=Y,
\end{equation}%
we would like to compute the Darboux derivative $\widetilde{F}^{\ast }\omega
$ of $\widetilde{F}$: By the moving frame formula (\ref{movingframe}), we
see that
\begin{equation}
\begin{split}
dF(u,v)& =X(\widetilde{F}^{\ast }\omega ^{1})+Y(\widetilde{F}^{\ast }\omega
^{2})+T(\widetilde{F}^{\ast }\omega ^{3}) \\
& =F_{u}du+F_{v}dv.
\end{split}%
\end{equation}%
This implies that
\begin{equation}
\begin{split}
F_{u}& =dF(\frac{\partial }{\partial u})=X(\widetilde{F}^{\ast }\omega ^{1})(%
\frac{\partial }{\partial u})+Y(\widetilde{F}^{\ast }\omega ^{2})(\frac{%
\partial }{\partial u})+T(\widetilde{F}^{\ast }\omega ^{3})(\frac{\partial }{%
\partial u}); \\
F_{v}& =dF(\frac{\partial }{\partial v})=X(\widetilde{F}^{\ast }\omega ^{1})(%
\frac{\partial }{\partial v})+Y(\widetilde{F}^{\ast }\omega ^{2})(\frac{%
\partial }{\partial v})+T(\widetilde{F}^{\ast }\omega ^{3})(\frac{\partial }{%
\partial v}),
\end{split}%
\end{equation}%
hence, comparing the coefficients and note that $F_{u}=X$, we have
\begin{equation}
(\widetilde{F}^{\ast }\omega ^{1})(\frac{\partial }{\partial u})=1,\  \ (%
\widetilde{F}^{\ast }\omega ^{2})(\frac{\partial }{\partial u})=(\widetilde{F%
}^{\ast }\omega ^{3})(\frac{\partial }{\partial u})=0,  \label{coeff1}
\end{equation}%
and
\begin{equation}
\begin{split}
(\widetilde{F}^{\ast }\omega ^{1})(\frac{\partial }{\partial v})&
=<F_{v},X>=a \\
(\widetilde{F}^{\ast }\omega ^{2})(\frac{\partial }{\partial v})&
=<F_{v},Y>=b \\
(\widetilde{F}^{\ast }\omega ^{3})(\frac{\partial }{\partial v})&
=<F_{v},T>=c.
\end{split}
\label{coeff2}
\end{equation}%
From (\ref{coeff1}) and (\ref{coeff2}), we get
\begin{equation}
\begin{split}
\widetilde{F}^{\ast }\omega ^{1}& =du+adv \\
\widetilde{F}^{\ast }\omega ^{2}& =bdv \\
\widetilde{F}^{\ast }\omega ^{3}& =cdv.
\end{split}
\label{Darder1}
\end{equation}%
On the other hand, again using the moving frame formula (\ref{movingframe}),
\begin{equation}
\begin{split}
dX(u,v)& =Y(\widetilde{F}^{\ast }\omega _{1}^{2})+T(\widetilde{F}^{\ast
}\omega ^{2}) \\
& =(\widetilde{F}^{\ast }\omega _{1}^{2})(\frac{\partial }{\partial u})Ydu+(%
\widetilde{F}^{\ast }\omega _{1}^{2})(\frac{\partial }{\partial v})Ydv+bTdv.
\end{split}%
\end{equation}%
Note again that $X=F_{u}$, we have
\begin{equation}
dX(u,v)=dF_{u}(u,v)=F_{uu}du+F_{uv}dv.
\end{equation}%
Comparing the above two formulae, we obtain
\begin{equation}
\begin{split}
(\widetilde{F}^{\ast }\omega _{1}^{2})(\frac{\partial }{\partial u})&
=<F_{uu},Y>=l \\
(\widetilde{F}^{\ast }\omega _{1}^{2})(\frac{\partial }{\partial v})&
=<F_{uv},Y>=m \\
b& =<F_{uv},T> \\
0& =<F_{uv},X>=<F_{uu},X>=<F_{uu},T>.
\end{split}
\label{Darder2}
\end{equation}%
In particular, combining (\ref{Darder1}) and (\ref{Darder2}), we get the
Darboux derivative $\widetilde{F}^{\ast }\omega $ which is
\begin{equation}
\widetilde{F}^{\ast }\omega =\left(
\begin{array}{cccc}
0 & 0 & 0 & 0 \\
du+adv & 0 & -ldu-mdv & 0 \\
bdv & ldu+mdv & 0 & 0 \\
cdv & bdv & -du-adv & 0%
\end{array}%
\right) .  \label{Darder9}
\end{equation}%
This completes the proof of the uniqueness. Now we prove the existence.
Suppose $a,b,c$ and $m,l$ are functions defined on $U$. Define a $psh(1)$%
-valued one form $\phi $ by
\begin{equation}
\phi =\left(
\begin{array}{cccc}
0 & 0 & 0 & 0 \\
du+adv & 0 & -ldu-mdv & 0 \\
bdv & ldu+mdv & 0 & 0 \\
cdv & bdv & -du-adv & 0%
\end{array}%
\right) .
\end{equation}%
Then we have
\begin{equation}
d\phi =\left(
\begin{array}{cccc}
0 & 0 & 0 & 0 \\
\frac{\partial a}{\partial u} & 0 & \frac{\partial l}{\partial v}-\frac{%
\partial m}{\partial u} & 0 \\
\frac{\partial b}{\partial u} & -\frac{\partial l}{\partial v}+\frac{%
\partial m}{\partial u} & 0 & 0 \\
\frac{\partial c}{\partial u} & \frac{\partial b}{\partial u} & -\frac{%
\partial a}{\partial u} & 0%
\end{array}%
\right) du\wedge dv
\end{equation}%
and
\begin{equation}
\phi \wedge \phi =\left(
\begin{array}{cccc}
0 & 0 & 0 & 0 \\
-lb & 0 & 0 & 0 \\
al-m & 0 & 0 & 0 \\
-2b & -m+al & bl & 0%
\end{array}%
\right) du\wedge dv.
\end{equation}%
Therefore we get that $\phi $ satisfies the integrability condition $d\phi
=-\phi \wedge \phi $ if and only lf $a,b,c,l$ and $m$ satisfy the
integrability condition (\ref{intcons1}). Therefore, by Theorem \ref{ft2},
there exists a map
\begin{equation*}
\widetilde{F}^{\ast }(u,v)=\left( F(u,v),X(u,v),Y(u,v),T\right)
\end{equation*}%
such that $\widetilde{F}^{\ast }\omega =\phi $. Thus, by the moving frame
formula (\ref{movingframe}), we see that $F:U\rightarrow H^{1}$ is a map
with $a,b,c,l$ and $m$ as its coefficients.

\subsection{Invariants of surfaces}

Let $\Sigma \hookrightarrow H^{1}$ be a surface such that each point of $%
\Sigma $ is regular. For each point $p\in \Sigma $, one can choose a
parametrization $F:U\rightarrow \Sigma $ with coordinates $(u,v)$ such that
\begin{equation}
F_{u}=\frac{\partial F}{\partial u}=X,
\end{equation}%
where $X$ is an unit vector field defining the characteristic foliation
around $p$. We call $F$ and $(u,v)$ a normal parametrization and a normal
coordinates around $p$, respectively.

\begin{lem}
\label{norcor} The normal coordinates is determined up to a transformation
of the form
\begin{equation}
\begin{split}
\widetilde{u}& =\pm u+g(v) \\
\widetilde{v}& =h(v),
\end{split}
\label{tlofnormal}
\end{equation}%
for some smooth functions $g(v),h(v)$ such that $\frac{\partial h}{\partial v%
}\neq 0$.
\end{lem}

\begin{proof}
Suppose that $(\widetilde{u},\widetilde{v})$ is another normal coordinates
around $p$, i.e.,
\begin{equation}
F_{\widetilde{u}}=\widetilde{X},
\end{equation}
where $\widetilde{X}=\pm X$. We have
\begin{equation}  \label{tlofnorpar}
\begin{split}
F_{u}&=F_{\widetilde{u}}\frac{\partial \widetilde{u}}{\partial u}+F_{%
\widetilde{v}}\frac{\partial \widetilde{v}}{\partial u} \\
F_{v}&=F_{\widetilde{u}}\frac{\partial \widetilde{u}}{\partial v}+F_{%
\widetilde{v}}\frac{\partial \widetilde{v}}{\partial v}.
\end{split}%
\end{equation}
Expand $F_{\widetilde{v}}=\widetilde{a}\widetilde{X}+\widetilde{b}\widetilde{%
Y}+\widetilde{c}\widetilde{T}$. By the first identity of (\ref{tlofnorpar}),
we have
\begin{equation}
\begin{split}
X&=\widetilde{X}\frac{\partial \widetilde{u}}{\partial u}+\left(\widetilde{a}%
\frac{\partial \widetilde{v}}{\partial u}\widetilde{X} +\widetilde{b}\frac{%
\partial \widetilde{v}}{\partial u}\widetilde{Y}+\widetilde{c}\frac{\partial
\widetilde{v}}{\partial u}\widetilde{T}\right) \\
&=\left(\frac{\partial \widetilde{u}}{\partial u}+\widetilde{a}\frac{%
\partial \widetilde{v}}{\partial u}\right)\widetilde{X}+ \widetilde{b}\frac{%
\partial \widetilde{v}}{\partial u}\widetilde{Y}+\widetilde{c}\frac{\partial
\widetilde{v}}{\partial u}\widetilde{T}.
\end{split}%
\end{equation}
Since $p$ is regular, we see that $\widetilde{c}\neq 0$ around $p$, we
conclude from the above formula
\begin{equation}
\frac{\partial \widetilde{v}}{\partial u}=0,\  \text{that is,}\  \  \widetilde{v%
}=h(v),
\end{equation}
for some function $h(v)$. In addition, comparing the coefficient of $X$, we
have
\begin{equation}
\pm 1=\frac{\partial \widetilde{u}}{\partial u}+\widetilde{a}\frac{\partial
\widetilde{v}}{\partial u}=\frac{\partial \widetilde{u}}{\partial u},
\end{equation}
hence $\widetilde{u}=\pm u+g(v)$ for some function $g(v)$. Finally we
compute
\begin{equation}
\det{\left(%
\begin{array}{cc}
\frac{\partial \widetilde{u}}{\partial u} & \frac{\partial \widetilde{u}}{%
\partial v} \\
\frac{\partial \widetilde{v}}{\partial u} & \frac{\partial \widetilde{v}}{%
\partial v}%
\end{array}%
\right)}=\det{\left(%
\begin{array}{cc}
\pm 1 & \frac{\partial g}{\partial v} \\
0 & \frac{\partial h}{\partial v}%
\end{array}%
\right)}=\pm \frac{\partial h}{\partial v}\neq 0.
\end{equation}
This completes the proof.
\end{proof}

Recall that by means of a normal parametrization $F$, we compute the Darboux
derivative $\widetilde{F}^{*}\omega$ as (\ref{Darder9}). One can define four
one-forms on $\Sigma$ locally as follows:
\begin{equation}  \label{ffs}
\begin{array}{ccl}
I=\widetilde{F}^{*}\omega^{1}=du+adv, & II=\widetilde{F}^{*}\omega^{2}=bdv,
& III=\widetilde{F}^{*}\omega^{3}=cdv \\
IV=\widetilde{F}^{*}\omega_{1}^{2}=ldu+mdv, &  &
\end{array}%
\end{equation}
where functions $a, b, c, m$ and $l$ are defined as (\ref{coeofform}). Let $(%
\widetilde{u},\widetilde{v})$ be another normal coordinates around $p$, we
have

\begin{prop}
\label{propofform}
\begin{equation*}
\widetilde{I}=\pm I,\  \widetilde{II}=\pm II,\  \widetilde{III}=III,\  \text{and%
}\  \widetilde{IV}=IV.
\end{equation*}
\end{prop}

\begin{proof}
From the definition of normal coordinates, we see that $F_{\widetilde{u}}=%
\widetilde{X}=\pm X$. By definition
\begin{equation}
\begin{array}{ccl}
\widetilde{I}=d\widetilde{u}+\widetilde{a}d\widetilde{v} & \widetilde{II}=%
\widetilde{b}d\widetilde{v} & \widetilde{III}=\widetilde{c}d\widetilde{v} \\
\widetilde{IV}=\widetilde{l}d\widetilde{u}+\widetilde{m}d\widetilde{v}, &  &
\end{array}%
\end{equation}
where
\begin{equation}
\widetilde{a}=<F_{\widetilde{v}},\widetilde{X}>,\  \widetilde{b}=<F_{%
\widetilde{v}},\widetilde{Y}>,\  \widetilde{c}=<F_{\widetilde{v}},T>,
\end{equation}
and
\begin{equation}
\widetilde{l}=<F_{\widetilde{u}\widetilde{u}},\widetilde{Y}>,\  \widetilde{m}%
=<F_{\widetilde{u}\widetilde{v}},\widetilde{Y}>, \widetilde{Y}=J_{0}%
\widetilde{X}=\pm Y.
\end{equation}
By lemma \ref{norcor}, there exists functions $g(v)$ and $h(v)$ such that
\begin{equation}
\begin{split}
\widetilde{u}&=\pm u+g(v) \\
\widetilde{v}&=h(v),
\end{split}%
\end{equation}
We compute the transformation laws of the coefficients of the fundamental
forms:
\begin{equation}  \label{tlofc}
\begin{split}
a&=<F_{v},X>=<F_{\widetilde{u}}\frac{\partial \widetilde{u}}{\partial v}+F_{%
\widetilde{v}}\frac{\partial \widetilde{v}}{\partial v},X> \\
&=<\pm X\frac{\partial g}{\partial v}+F_{\widetilde{v}}\frac{\partial h}{%
\partial v},X> \\
&=\pm \left(\frac{\partial g}{\partial v}+\frac{\partial h}{\partial v}%
\widetilde{a}\right).
\end{split}%
\end{equation}
Similarly, we have
\begin{equation}  \label{tlofc1}
b=\pm \frac{\partial h}{\partial v}\widetilde{b},\ c=\frac{\partial h}{%
\partial v}\widetilde{c}.
\end{equation}
On the other hand, note that $F_{u}=\pm F_{\widetilde{u}}$, hence $%
F_{uu}=\pm(F_{\widetilde{u}\widetilde{u}}\frac{\partial \widetilde{u}}{%
\partial u} +F_{\widetilde{u}\widetilde{v}}\frac{\partial \widetilde{v}}{%
\partial u})=F_{\widetilde{u}\widetilde{u}}$. Thus
\begin{equation}  \label{tlofc2}
l=\pm \widetilde{l}.
\end{equation}
Similarly we have
\begin{equation}  \label{tlofc3}
m=\frac{\partial g}{\partial v}\widetilde{l}+\frac{\partial h}{\partial v}%
\widetilde{m}.
\end{equation}
From the transformation laws (\ref{tlofc}), (\ref{tlofc1}), (\ref{tlofc2})
and (\ref{tlofc3}), it is easy to see that
\begin{equation*}
\widetilde{I}=\pm I,\  \widetilde{II}=\pm II,\  \widetilde{III}=III,\  \text{and%
}\  \widetilde{IV}=IV.
\end{equation*}
This finishes the proof of the proposition.
\end{proof}

Define $\alpha=\frac{b}{c}$ and $\widetilde{\alpha}=\frac{\widetilde{b}}{%
\widetilde{c}}$, then from (\ref{tlofc1}), we see that $\alpha=\pm
\widetilde{\alpha}$. Actually, $\alpha$ is the function defined on the
non-singular part of $\Sigma$ such that $\alpha e_{2}+T\in T\Sigma$. Up to a
sign, $\alpha $ is a function which is independent of the choice of the
normal coordinates, hence an invariant of $\Sigma$ on the non-singular part.
Similarly, from (\ref{tlofc2}), so is for $l$, which actually is the $p$%
-mean curvature.

\begin{rem}
Note that if we restrict us to choose normal coordinates with respect to a
fixed orientation of the characteristic foliation on the nonsingular part,
we see, from the proof of Proposition \ref{propofform}, that $\alpha =%
\widetilde{\alpha }$ and $l=\widetilde{l}$. That is, the sign appearing is
due to the different choice of orientation.
\end{rem}

Besides the two invariants $\alpha$ and $l$, we now proceed to introduce
another invariant of $\Sigma$, which is defined on all of $\Sigma$, not just
on the non-singular part. Again, from Proposition \ref{propofform}, it is
easy to see that
\begin{equation}
I\otimes I+II\otimes II+III\otimes III=\widetilde{I}\otimes \widetilde{I}+%
\widetilde{II}\otimes \widetilde{II}+\widetilde{III}\otimes \widetilde{III}.
\end{equation}
Therefore the form $I\otimes I+II\otimes II+III\otimes III$ is again
independent of the choice of a normal coordinates, hence also an invaiant of
$\Sigma$

\begin{lem}
Let $g_{\theta_{0}}$ be the adapted metric on $H^1$. Then we have
\begin{equation}
g_{\theta_{0}}|_{\Sigma}=I\otimes I+II\otimes II+III\otimes III,
\end{equation}
on the non-singular part of $\Sigma$.
\end{lem}

\begin{proof}
This lemma is a easy consequence of the first one of the moving frame
formula (\ref{movingframe}).
\end{proof}

In the following section, we will show that the form $IV=\widetilde{F}^{\ast
}\omega _{1}^{2}$ is completely determined by all $g_{\theta _{0}},\alpha $
and $l$. We therefore obtain a complete set of invariants for surfaces on
the non-singular part.

\section{A complete set of invariants for surfaces in $H^1$}

Let $\Sigma $ be an oriented surface and suppose $f:\Sigma \rightarrow H^{1}$
be an embedding. For the convenient of expression, we will not distinguish
surfaces $\Sigma $ and $f(\Sigma )$. For each non-singular point $p\in
\Sigma $, we specify an orthonormal frame by $(p;e_{1},e_{2},T)$, here $%
e_{1} $ is tangent to the characteristic foliation and $e_{2}=J_{0}e_{1}$. A
Darboux frame is a moving frame which is smoothly defined on $\Sigma $,
except those singular points, hence giving a lifting of $f$ to $PSH(1)$
which is defined by $F$. Now we would like to compute the Darboux derivative
$F^{\ast }\omega $ of $F$. In the following, instead of $F^{\ast }\omega $,
we still use
\begin{equation}
\omega =\left(
\begin{array}{cccc}
0 & 0 & 0 & 0 \\
\omega ^{1} & 0 & -\omega _{1}{}^{2} & 0 \\
\omega ^{2} & \omega _{1}{}^{2} & 0 & 0 \\
\omega ^{3} & \omega ^{2} & -\omega ^{1} & 0%
\end{array}%
\right) ,
\end{equation}%
to express the Darboux derivative. It satisfies the integrability condition $%
d\omega +\omega \wedge \omega =0$, that is,
\begin{equation}
\begin{split}
d\omega ^{1}& =\omega _{1}^{2}\wedge \omega ^{2} \\
d\omega ^{2}& =-\omega _{1}^{2}\wedge \omega ^{1} \\
d\omega ^{3}& =2\  \omega ^{1}\wedge \omega ^{2} \\
d\omega _{1}^{2}& =0
\end{split}
\label{mcsteq}
\end{equation}

Let $g_{\theta_{0}}=h+\theta_{0}^{2}$ be the adapted metric. From Section %
\ref{invpasur}, which we see that $\omega^{2}=\alpha \omega^{3}$ on the
nonsingular part of $\Sigma$, it is easy to see that
\begin{equation*}
\begin{split}
g_{\theta_{0}}|_{\Sigma}&=\omega^{1}\otimes \omega^{1}+\omega^{2}\otimes
\omega^{2}+\omega^{3}\otimes \omega^{3} \\
&=\omega^{1}\otimes \omega^{1}+(1+\alpha^2)\omega^{3}\otimes \omega^{3}.
\end{split}%
\end{equation*}
Define
\begin{equation}  \label{baeq1}
\begin{split}
\hat{\omega}^{1}&=\omega^{1} \\
\hat{\omega}^{2}&=\sqrt{1+\alpha^{2}}\omega^{3}.
\end{split}%
\end{equation}
This is an orthonormal coframe of $g_{\theta_{0}}|_{\Sigma}$ and the dual
frame is
\begin{equation}
\begin{split}
\hat{e}_{1}&=e_{1} \\
\hat{e}_{2}&=e_{\Sigma}=\frac{\alpha e_{2}+T}{\sqrt{1+\alpha^{2}}}.
\end{split}%
\end{equation}
Let $\hat{\omega}_{1}^{2}$ be the Levi-Civita connection of $%
g_{\theta_{0}}|_{\Sigma}$ with respect to the frame $\hat{\omega}^{1},\hat{%
\omega}^{2}$. By the fundamental theorem in Riemannian geometry, this
connection is uniquely defined by
\begin{equation}  \label{risteq}
\begin{split}
d\hat{\omega}^{1}&=-\hat{\omega}_{2}^{1}\wedge \hat{\omega}^{2} \\
d\hat{\omega}^{2}&=-\hat{\omega}_{1}^{2}\wedge \hat{\omega}^{1} \\
\hat{\omega}_{1}^{2}&=-\hat{\omega}_{2}^{1}.
\end{split}%
\end{equation}

The following Proposition point out that $\omega_{1}^{2}$ is completely
determined by the induced fundamental form $g_{\theta_{0}}|_{\Sigma}$ and
the functions $\alpha$ and $l$.

\begin{prop}
\label{prop1} We have
\begin{equation}
\begin{split}
\omega_{1}^{2}&=\frac{\alpha}{\sqrt{1+\alpha^{2}}}\hat{\omega_{1}^{2}}+\frac{%
l}{1+\alpha^{2}}\hat{\omega}^{1}+ \frac{e_{1}\alpha}{(1+\alpha^2)^{\frac{3}{2%
}}}\hat{\omega}^2. \\
&=l\hat{\omega}^1+\frac{2\alpha^{2}+(e_{1}\alpha)}{\sqrt{1+\alpha^{2}}}\hat{%
\omega}^2, \\
& \\
\hat{\omega}_{1}^{2}&=\frac{\alpha}{\sqrt{1+\alpha^{2}}}\omega_{1}^{2}+\frac{%
2\alpha}{1+\alpha^2}\hat{\omega}^{2} \\
&=\frac{l\alpha}{\sqrt{1+\alpha^{2}}}\hat{\omega}^{1}+\left(2\alpha+\frac{%
\alpha(e_{1}\alpha)}{1+\alpha^{2}}\right)\hat{\omega}^{2}.
\end{split}%
\end{equation}
\end{prop}

\begin{proof}
Note that $\omega ^{2}=\alpha \omega ^{3}$. Then from the second identity of
(\ref{baeq1}), we have
\begin{equation*}
\begin{split}
d\omega ^{2}& =d\left( \frac{\alpha }{(1+\alpha ^{2})^{\frac{1}{2}}}\hat{%
\omega}^{2}\right) \\
& =d\left( \frac{\alpha }{(1+\alpha ^{2})^{\frac{1}{2}}}\right) \wedge \hat{%
\omega}^{2}+\frac{\alpha }{(1+\alpha ^{2})^{\frac{1}{2}}}d\hat{\omega}^{2} \\
& =e_{1}\left( \frac{\alpha }{(1+\alpha ^{2})^{\frac{1}{2}}}\right) \hat{%
\omega}^{1}\wedge \hat{\omega}^{2}-\frac{\alpha }{(1+\alpha ^{2})^{\frac{1}{2%
}}}\hat{\omega}_{1}^{2}\wedge \hat{\omega}^{1} \\
& =\hat{\omega}^{1}\wedge \left( e_{1}\left( \frac{\alpha }{(1+\alpha ^{2})^{%
\frac{1}{2}}}\right) \hat{\omega}^{2}+\frac{\alpha }{(1+\alpha ^{2})^{\frac{1%
}{2}}}\hat{\omega}_{1}^{2}\right) ,
\end{split}%
\end{equation*}%
where at the third equality above, we have used the second formula of the
structure equation (\ref{risteq}) in Riemannian geometry. On the other hand,
from the Maurer-Cartan structure equation (\ref{mcsteq})
\begin{equation*}
d\omega ^{2}=-\omega _{1}^{2}\wedge \omega ^{1}=\hat{\omega}^{1}\wedge
\omega _{1}^{2}.
\end{equation*}%
Together the above two formulae and by Cartan lemma, we see that there
exists a function $D$ such that
\begin{equation}
\begin{split}
\omega _{1}^{2}& =e_{1}\left( \frac{\alpha }{(1+\alpha ^{2})^{\frac{1}{2}}}%
\right) \hat{\omega}^{2}+\frac{\alpha }{(1+\alpha ^{2})^{\frac{1}{2}}}\hat{%
\omega}_{1}^{2}+D\hat{\omega}^{1} \\
& =\frac{e_{1}\alpha }{(1+\alpha ^{2})^{\frac{3}{2}}}\hat{\omega}^{2}+\frac{%
\alpha }{(1+\alpha ^{2})^{\frac{1}{2}}}\hat{\omega}_{1}^{2}+D\hat{\omega}%
^{1}.
\end{split}
\label{baeq2}
\end{equation}%
Similarly, we compute
\begin{equation}
\begin{split}
-\hat{\omega}_{2}^{1}\wedge \hat{\omega}^{2}& =d\hat{\omega}^{1}=d\omega ^{1}
\\
& =\omega _{1}^{2}\wedge \omega ^{2} \\
& =\frac{\alpha }{\sqrt{1+\alpha ^{2}}}\omega _{1}^{2}\wedge \hat{\omega}%
^{2}.
\end{split}%
\end{equation}%
Again, by Cartan lemma, there exists a function $A$ such that
\begin{equation}
-\hat{\omega}_{2}^{1}=\frac{\alpha }{\sqrt{1+\alpha ^{2}}}\omega _{1}^{2}+A%
\hat{\omega}^{2}.  \label{baeq3}
\end{equation}%
Finally, we compute
\begin{equation}
\begin{split}
-\hat{\omega}_{1}^{2}\wedge \hat{\omega}^{1}& =d\hat{\omega}^{2}=d\left(
(1+\alpha ^{2})^{\frac{1}{2}}\omega ^{3}\right) \\
& =(1+\alpha ^{2})^{\frac{1}{2}}d\omega ^{3}+d(1+\alpha ^{2})^{\frac{1}{2}%
}\wedge \omega ^{3} \\
& =2\alpha (1+\alpha ^{2})^{\frac{1}{2}}\hat{\omega}^{1}\wedge \omega ^{3}+%
\frac{\alpha }{(1+\alpha ^{2})^{\frac{1}{2}}}d\alpha \wedge \omega ^{3} \\
& =\left( 2\alpha +\frac{\alpha (e_{1}\alpha )}{1+\alpha ^{2}}\right) \hat{%
\omega}^{1}\wedge \hat{\omega}^{2},
\end{split}%
\end{equation}%
where we have used the third formula of (\ref{mcsteq}) and $\hat{\omega}%
^{2}\wedge \omega ^{3}=0$. Therefore, there exists a function $B$ such that
\begin{equation}
\hat{\omega}_{1}^{2}=\left( 2\alpha +\frac{\alpha (e_{1}\alpha )}{1+\alpha
^{2}}\right) \hat{\omega}^{2}+B\hat{\omega}^{1}.  \label{baeq4}
\end{equation}%
By (\ref{baeq2}) and (\ref{baeq3}), we get
\begin{equation*}
\begin{split}
D& =\omega _{1}^{2}(e_{1})-\frac{\alpha }{\sqrt{1+\alpha ^{2}}}\hat{\omega}%
_{1}^{2}(e_{1}) \\
& =\frac{\omega _{1}^{2}(e_{1})}{1+\alpha ^{2}}=\frac{l}{1+\alpha ^{2}}.
\end{split}%
\end{equation*}%
Similarly, by (\ref{baeq2}), (\ref{baeq3}) and (\ref{baeq4}), we obtain
\begin{equation}
\begin{split}
A& =\frac{2\alpha }{1+\alpha ^{2}} \\
B& =\frac{l\alpha }{\sqrt{1+\alpha ^{2}}}.
\end{split}%
\end{equation}%
These complete the proof.
\end{proof}

\subsection{The proof of Theorem \protect \ref{main6}}

Let $K$ be the Gaussian curvature of the induced metric $g_{\theta_{0}}|_{%
\Sigma}$, hence we have
\begin{equation}
d\hat{\omega}_{1}^{2}=Kd\sigma,
\end{equation}
where $d\sigma$ is the area form $\hat{\omega}^{1}\wedge \hat{\omega}^{2}$.
Using Proposition \ref{prop1} and (\ref{mcsteq}) and (\ref{risteq}), we
compute
\begin{equation}
\begin{split}
d\hat{\omega}_{1}^{2}&=d\left(\frac{\alpha}{\sqrt{1+\alpha^{2}}}%
\omega_{1}^{2}+\frac{2\alpha}{1+\alpha^2}\hat{\omega}^{2}\right) \\
&=d\left(\frac{\alpha}{\sqrt{1+\alpha^{2}}}\right)\wedge
\omega_{1}^{2}+d\left(\frac{2\alpha}{1+\alpha^2}\right)\wedge \hat{\omega}%
^{2} +\frac{2\alpha}{1+\alpha^2}d\hat{\omega}^{2} \\
&=\frac{d\alpha}{(1+\alpha^2)^{\frac{3}{2}}}\wedge \omega_{1}^{2}+\frac{%
2(1-\alpha^2)d\alpha}{(1+\alpha^2)^{2}}\wedge \hat{\omega}^{2} -\frac{2\alpha%
}{1+\alpha^2}\hat{\omega}_{1}^{2}\wedge \hat{\omega}^{1} \\
&=\left(\frac{(e_{1}\alpha)^{2}+2(1+\alpha^{2})(e_{1}\alpha)+4\alpha^{2}(1+%
\alpha^{2})-l(e_{\Sigma}\alpha)(1+\alpha^{2})^{\frac{1}{2}}} {%
(1+\alpha^{2})^{2}}\right)\hat{\omega}^{1}\wedge \hat{\omega}^{2}.
\end{split}%
\end{equation}
These completes the proof of Theorem \ref{main6}.

\subsection{The derivation of the integrability condition (\protect \ref%
{Intconsur})}

We compute
\begin{equation}
\begin{split}
0&=d\omega_{1}^{2} \\
&=d\left(\frac{\alpha}{\sqrt{1+\alpha^{2}}}\hat{\omega_{1}^{2}}+\frac{l}{%
1+\alpha^{2}}\hat{\omega}^{1}+ \frac{e_{1}\alpha}{(1+\alpha^2)^{\frac{3}{2}}}%
\hat{\omega}^2\right) \\
&=\Big \{-(1+\alpha^{2})^{\frac{3}{2}}(e_{\Sigma}l)+(1+%
\alpha^{2})(e_{1}e_{1}\alpha)-\alpha(e_{1}\alpha)^{2}+4\alpha(1+%
\alpha^{2})(e_{1}\alpha) \\
&+\alpha(1+\alpha^{2})^{2}K+\alpha l(1+\alpha^{2})^{\frac{1}{2}%
}(e_{\Sigma}\alpha)+\alpha(1+\alpha^{2})l^{2}\Big \} \frac{\hat{\omega}%
^{1}\wedge \hat{\omega}^{2}}{(1+\alpha^2)^{\frac{5}{2}}}.
\end{split}%
\end{equation}
Therefore the integrability condition (\ref{Intconsur}) is equivalent to $%
d\omega_{1}^{2}=0$.

\subsection{The proof of Theorem \protect \ref{main8}}

First we show the existence. Define an $psh(1)$-valued one-form $\phi$ on
the non-singular part of $\Sigma$ by
\begin{equation}
\phi=\left(%
\begin{array}{cccc}
0 & 0 & 0 & 0 \\
\hat{\omega}^{1} & 0 & -\omega_{1}^{2} & 0 \\
\frac{\alpha^{^{\prime }}}{\sqrt{1+(\alpha^{^{\prime }})^{2}}}\hat{\omega}%
^{2} & \omega_{1}^{2} & 0 & 0 \\
\frac{1}{\sqrt{1+(\alpha^{^{\prime }})^{2}}}\hat{\omega}^{2} & \frac{%
\alpha^{^{\prime }}}{\sqrt{1+(\alpha^{^{\prime }})^{2}}}\hat{\omega}^{2} & -%
\hat{\omega}^{1} & 0%
\end{array}%
\right),
\end{equation}
where
\begin{equation}
\omega_{1}^{2}=\frac{\alpha^{^{\prime }}}{\sqrt{1+(\alpha^{^{\prime }})^{2}}}%
\hat{\omega_{1}^{2}}+\frac{l^{^{\prime }}}{1+(\alpha^{^{\prime }})^{2}}\hat{%
\omega}^{1}+ \frac{e_{1}\alpha^{^{\prime }}}{(1+(\alpha^{^{\prime }})^{2})^{%
\frac{3}{2}}}\hat{\omega}^2.
\end{equation}
Then it is easy to check that $\phi$ satisfies $d\phi+\phi \wedge \phi=0$ if
and only if the integrability condition (\ref{Intconsur}) holds. Therefore,
by Theorem \ref{ft2}, for each $x\in \Sigma$ there exists an open set $U$
containing $x$ and an embedding $f:U\rightarrow H^1$ such that $%
g=f^{*}(g_{\theta_{0}}), \alpha^{^{\prime }}=f^{*}\alpha$ and $l^{^{\prime
}}=f^{*}l$. Next we show the uniqueness. By Proposition \ref{prop1}, we see
that the Darboux derivative is completely determined by the induced metric $%
g_{\theta_{0}}|_{\Sigma}$, the $p$-variation $\alpha$ and the $p$-mean
curvature $l$. Therefore, by Theorem \ref{ft1}, the embedding into $H^1$ is
unique up to a Heisenberg rigid motion.

\section{Appendix}

In this Appendix, we give another proof of Theorem \ref{main1}.

\subsection{The second proof of Theorem \protect \ref{main1}}

For a horizontally regular curve $\gamma(s)$ parametrized by horizontal
arc-length $s$, we define a moving frames $(X(s),Y(s),T(s))$ by
\begin{equation}  \label{mfs}
X(s)=\gamma_{\xi}^{^{\prime }}(s),\  \ Y(s)=JX(s),\  \  \text{and}\  \ T(s)=T.
\end{equation}
Then we have that
\begin{equation}  \label{mfsf1}
\begin{split}
X^{^{\prime }}(s)=k(s)Y(s) \\
Y^{^{\prime }}(s)=-k(s)X(s)-T \\
T^{^{\prime }}(s)=0.
\end{split}%
\end{equation}
Note also that
\begin{equation}  \label{mfsf2}
\gamma^{^{\prime }}(s)=X(s)+\tau(s)T.
\end{equation}
Now, assume that two curves $\gamma(s)$ and $\bar{\gamma}(s)$ satisfy the
conditions
\begin{equation}
k(s)=\bar{k}(s)\  \  \text{and}\  \  \tau(s)=\bar{\tau}(s),\ s\in I.
\end{equation}
After performing a Heisenberg rigid motion (i.e., a pseudohermitian
transformation on $H^1$), we can assume, without loss of generality, that
\begin{equation}
\bar{\gamma}(s_0)=\gamma(s_{0}), \bar{X}(s_0)=X(s_0),\  \text{and}\  \bar{Y}%
(s_0)=Y(s_0),
\end{equation}
for a fixed $s_{0}\in I$. Define $A(s)=<X(s),\bar{X}(s)>+<Y(s),\bar{Y}(s)>$.
By using the moving franes formula (\ref{mfsf1}), we have
\begin{equation}
\begin{split}
A^{^{\prime }}(s)&=<X^{^{\prime }}(s),\bar{X}(s)>+<X(s),\bar{X}^{^{\prime
}}(s)>+<Y^{^{\prime }}(s),\bar{Y}(s)>+<Y(s),\bar{Y}^{^{\prime }}(s)> \\
&=k<Y(s),\bar{X}(s)>+\bar{k}<X(s),\bar{Y}(s)>+<-kX-T,\bar{Y}(s)>+<Y,-\bar{k}%
\bar{X}-\bar{T}> \\
&=0.
\end{split}%
\end{equation}
Since $A(s_0)=2$, we get $A(s)=2$, hence that $X(s)=\bar{X}(s)$ and $Y(s)=%
\bar{Y}(s)$ for each $s\in I$. In particular, we have $\gamma_{\xi}^{^{%
\prime }}(s)=\bar{\gamma}_{\xi}^{^{\prime }}(s)$. Also note that $\tau(s)=%
\bar{\tau}(s)$, by (\ref{mfsf2}), we have $\gamma_{T}^{^{\prime }}(s)=\bar{%
\gamma}_{T}^{^{\prime }}(s)$. We therefore obtain that $\gamma^{^{\prime
}}(s)=\bar{\gamma}^{^{\prime }}(s)$, which implies that $\gamma(s)=\bar{%
\gamma}(s)+C$ for some constant $C$. Since $\gamma(s_0)=\bar{\gamma}(s_0)$,
we see that $C=0$, that is, $\gamma(s)=\bar{\gamma}(s)$ for all $s\in I$.

%\bibliography{main}

\begin{thebibliography}{99}
\bibitem{CCL} Cheng, S.S., Cheng, W.H. and Lam, K.S.: Lecture notes on
Differential Geometry.;

\bibitem{C} Chevalley, C.,: Theory of Lie Groups (Princeton University
Press, Princeton, 1946).;

\bibitem{CCG} Calin, O., Chang, D.C. and Greiner, P.: Geometric Analysis on
the Heisenberg Group and Its Generalizations.;

\bibitem{CC} Calin, O., Chang, D.C.: Sub-Riemannian Geometry: General Theory
and Examples (Cambridge ; New York : Cambridge University Press, 2009).;

\bibitem{CHMY} Cheng,J.-H.; Hwang,J.-F.; Malchiodi, A. and Yang, P. : A
Codazzi-Like Equation and the Singular Set for $C^1$ Smooth surfaces In the
Heisenberg Group. \textit{J. reine angew. Math. }\textbf{671} (2012),
131-198;

\bibitem{G} Griffiths, P. : On Cartan's Method of Lie Groups and Moving
frames as applied to Uniqueness and Existence questions in Differential
Geometry,\textit{Duke Math. J.}\textbf{41} (1974), 775-814;

\bibitem{IL} Ivey, T.A. and Landsberg, J.M. : Cartan for
Beginners:Differential Geometry via Moving Frames and Exterior Differential
Systems. \textit{Graduate Studies, in Math.} \textbf{v.61} (2003);

\bibitem{Le1} Lee, J.M. : The Fefferman metric and pseudohermitian
invariants. \textit{Trans. Am. Math.Soc.} \textbf{296} (1986), 411-429 ;

\bibitem{Le2} Lee, J.M. : Pseudo-Einstein structures on CR manifolds.
\textit{Am. J. Math.} \textbf{110} (1988), 157-178 ;

\bibitem{S} Sharp, R.W. : \textbf{Differential Geometry }Cartan's
Generalization of Klein's Erlangen Program. \textit{Graduate Texts, in Math.}
{v.166} (1997);

\bibitem{We} Webster, S.M.. : Pseudo-Hermitian structures on a real
Hypersurface. \textit{J. Diff. Geom.} \textbf{13} (1978), 25-41;
\end{thebibliography}

\end{document}